\documentclass[11pt,twoside]{amsart}
\usepackage{mathtools,amsmath,amssymb}
\usepackage{mathrsfs}

\usepackage[a4paper,top=1.2in, bottom=1.2in, left=1.2in, right=1.2in]{geometry} 
\usepackage{graphicx} 
\usepackage[colorlinks=true,linkcolor=black,citecolor=black,urlcolor=black]{hyperref} 
\usepackage{cleveref}
\usepackage[backend=biber,backref=true,doi=false,url=false,maxnames=9,isbn=true]{biblatex} 
\AtEveryBibitem{\clearfield{pagination}}
\AtEveryBibitem{\clearfield{pages}}
\renewbibmacro*{pageref}{%
 \iflistundef{pageref}
   {}
   {\addspace
    \mkbibbrackets{\printlist[pageref][-\value{listtotal}]{pageref}}}} 
\usepackage{tikz-cd}
\usepackage{mathptmx}
\usepackage{float}
\usepackage{amssymb}
\usepackage[english]{babel}
\usepackage{amsthm}
\usepackage{amsmath}
\usepackage{mathtools}
\usepackage{soul}
\usepackage{array}   
\usepackage{float}
\usepackage{orcidlink}
\usepackage{nomencl}

\makenomenclature
\usepackage[toc,page]{appendix}
\usepackage{caption}
\RequirePackage{fix-cm}
\usepackage{longtable}
\usepackage{mathrsfs}     
\usepackage{csquotes}
\newcommand{\GL}{\mathrm{GL}}
\newcommand{\M}{\mathrm{M}}

\newcommand{\Z}{\mathbb{Z}}

\newcommand{\F}{\mathbb{F}}

\newcommand{\Zen}{\mathscr{Z}}

\newcommand{\diag}{\mathrm{diag}}

\newcommand{\per}{\operatorname{Per}}
\newcommand{\I}{\operatorname{I}}
\renewcommand{\d}{\mathfrak d}
\newcommand{\Sp}{\operatorname{Sp}}
\newcommand{\Un}{\operatorname{U}}
\newcommand{\Or}{\operatorname{O}}
\usepackage{thmtools, thm-restate}
\declaretheorem{theorem}
\newtheorem{lemma}{Lemma}[section]
\newtheorem{proposition}[lemma]{Proposition}

\theoremstyle{definition} 

\newtheorem{example}[lemma]{Example}

\title[Periodic points for power maps]{Periodic Points of Power Maps in Finite Matrix Groups and Algebras}
\date{\today}
\author[Saikat Panja]{Saikat Panja\orcidlink{0000-0002-9639-3122}}
\email{panjasaikat300@gmail.com}
\address{Indian Statistical Institute, Bengaluru Centre, 8th Mile, Mysore Rd, RVCE Post, Gnana Bharathi, Bengaluru, Karnataka 560059, India}
\thanks{Panja is supported by an NBHM postdoctoral fellowship, file number ending at R\&D-II/6746.}
\date{\today}
\begin{document}
\begin{abstract}
Consider the power map $x\mapsto x^L$ for a prime $L\neq 2$ such that $L|q-1$ where $q$ is a power of a prime.
We determine the periodic points under this map for $\M_n(q)$, the algebra of $n\times n $ matrices over a finite field of order $q$, and also for the group $\GL_n(q)=\M_n(q)^\times$.
We compute the limit $      \lim\limits_{\substack{q\longrightarrow \infty\\v_L(q-1)=c}}\dfrac{\left|\operatorname{Per}(x^L,\operatorname{M}_\ell(q))\right|}{|\operatorname{M}_\ell(q)|}$ and consequently $\lim\limits_{\substack{q\longrightarrow \infty\\
    v_L(q-1)=c}}\dfrac{\left|\operatorname{Per}(x^L,\operatorname{GL}_\ell(q))\right|}{|\operatorname{GL}_\ell(q)|}$, where $v_L$ denotes the $L$-adic valuation. We also compute the quantity $\lim\limits_{\substack{q\longrightarrow \infty\\
    v_L(q-1)=c}}\dfrac{\left|\operatorname{Per}(x^L,\operatorname{Sp}_{2\ell}(q))\right|}{|\operatorname{Sp}_{2\ell}(q)|}$ and $\lim\limits_{\substack{q\longrightarrow \infty\\
    v_L(q-1)=c}}\dfrac{\left|\operatorname{Per}(x^L,\operatorname{U}_\ell(q))\right|}{|\operatorname{U}_\ell(q)|}$; turns out these two limiting values are same.
    In all the cases it turns out that the regular semisimple elements play the role in determining the limiting values.
\end{abstract}
\maketitle
\section{Introduction}\label{sec:intro}
\subsection{Dynamical systems} A discrete dynamical system is a pair $(S,f)$ where $S$ is set and $f$ is a self-map. 
For a positive integer $n\geq 1$, fix the notation 
\begin{align*}
    f^n=\underbrace{f\circ f\circ \cdots \circ f}_{n\text{ times}}.
\end{align*}
A point $\alpha\in S$ is called \emph{periodic} if there exists $n>1$ such that $f^n(\alpha)=\alpha$.
The set of all periodic points of the system $(S,f)$ is denoted by $\per(f, S)$.
If we consider $S$ to be a finite field $\F_q$ and $f$ to be the polynomial $t^L$ for some $L\geq 2$ or the Chebyshev polynomial, then the arithmetic dynamics of this system is studied in \cite{ManesThompson2019}.
These questions have been studied from the viewpoints of arithmetic and graph theory in \cite{ShaHu2011}
The proportion of periodic points for a polynomial of the form $z^2+c$ is considered in \cite{Madhu2011}.
In this article, we replace the field by a finite $\F_q$-algebra or a group. 
Our motivation to address this problem comes from the subject area of word maps on groups and polynomial maps on algebras.
\subsection{Word maps on groups and polynomial maps on algebras}
Given a tuple $(w_1,w_2,\ldots,w_r)$ of words from the free group $\mathbf F_r$ and a group $G$, one obtains a discrete dynamical system of a tuple of word maps through evaluations
\begin{align*}
    (w_1,w_2,\ldots,w_r):G^r\longrightarrow G^r,\,\overline{g}=(g_1,\ldots, g_r)\mapsto(w_1(\overline{g}),\ldots, w_r(\overline{g})).
\end{align*}
By replacing the free group by a free $R$-algebra on $r$ indeterminates and replacing $G$ by an $R$-algebra $A$, one gets a dynamical system of a tuple of polynomial maps.

The study of word maps dates back to 1951, when {\O}ystein Ore proved that every element of the alternating group $A_n$ is a commutator and conjectured the same for all finite non-abelian simple groups \cite{Ore1951}. This conjecture was eventually proved through the work of many mathematicians; see \cite{LiebeckObrienShalevTiep2010} and the references therein. A related result shows that for every finite non-abelian simple group $G$ of sufficiently large order and every nontrivial word $w$, one has $w(G)^2 = G$ \cite{LarsenShalevTiep2011}. The exponent $2$ is best possible; for instance, the squaring map on $A_5$ is not surjective.

Polynomial maps have also been widely studied. Shoda proved in 1937 that over a field of characteristic zero every trace-zero matrix is a commutator \cite{Shoda1937}. Kaplansky and L'vov conjectured that the image of a multilinear polynomial on $\mathrm{M}_n(k)$ over an infinite field $k$ is always a vector space; despite partial progress \cite{KanelMalevRowen2012,KanelMalevRowen2016}, the conjecture remains open for $n \ge 4$. Br\v esar showed that if $C$ is a commutative unital algebra over a field $\mathbb F$ of characteristic $0$, $A=\M_n(C)$, and $f$ is neither an identity nor a central polynomial of $A$, then every commutator in $A$ is a difference of two sums of $7788$ elements from $f(A)$ \cite{Bresar2020}; see also \cite{BresarVolcic2025}. Polynomial maps have also been studied on upper triangular matrix algebras \cite{GargateDeMello2022,PanjaPrasad2023} and octonion algebras \cite{PanjaSainiSingh2025}; see also \cite{PanjaSainiSinghConst}.
\subsection{Power maps}
In this article, we consider the power maps on a finite matrix algebra $\M_n(q)$, of $n\times n$ matrices, and the finite general linear groups, symplectic groups and unitary groups.
We consider the power map, i.e., the map $x\mapsto x^L$ where $L\geq 2$ is a positive integer.
These maps have been investigated earlier by Chatterjee and Steinberg, independently, in the context of algebraic groups \cite{Chatterjee02,Steinberg03}; for some finite groups of Lie type in \cite{KunduSingh2024, PanjaSingh2025,PanjaSingh2024unitary}, for general linear groups over finite principal ideal local rings of length two in \cite{PanjaRoySingh2025}; see also \cite{Panja2025roots} and the references therein. 
Power maps exhibit several notable features (see \cite{Panja2024c}) and serve as a useful tool for studying the number of real conjugacy classes \cite{panja2024d}. 
A complex dynamical study has been carried out for the power map on $\GL_n(\mathbb C)$ and $\M_n(\mathbb C)$ has been carried out in \cite{panja2025dynamics}.

While dealing with the dynamical system $(\M_n(q),f)$, it is obvious that a matrix of the form $\diag(\alpha_1,\alpha_2,\cdots,\alpha_n)$, where all $\alpha_i$s are periodic points of $f$ with period $\ell$, is a member of $\per(\M_n(q),f)$. 
However, these are not all the points in $\per(\M_n(q),f)$; as shown by the following example.
\begin{example}
Consider the power map $x\mapsto x^2$, the algebra $\M_2(59)$ and the element 
$A=\begin{pmatrix}
    0 & 42\\1 & 31
\end{pmatrix}$.
Then one has that $A$ has an irreducible polynomial, and it lies in a $28$-cycle of the squaring map. 
So there are matrices which are not diagonal (or a conjugate of a diagonal matrix), but are periodic.    
\end{example}
With this in mind, we explore the periodic points for the dynamical system $(G_n(q),x^L)$ where $L$ is a prime, $\gcd(L,q)=1$, and $G_n(q)$ is a finite matrix algebra or a matrix group; more precisely one of finite general linear, symplectic or unitary groups.
We consider the limiting value
\begin{align*}
    \lim\limits_{n\longrightarrow\infty} \dfrac{|\per(G_n(q),x^L)|}{|G_n(q)|},
\end{align*}
but this limit does not exist in general; for example, if $x^L$ (say $(L,\M_n(q)=1$) is a permutation polynomial of $\M_n(q)$, then the ratio is $1$, whereas if it is not a permutation polynomial, the value need not be $1$.
Thus, one needs to put an extra condition on the $L$-adic evaluation of $q-1$. 
For example, we prove in \cref{thm:GL} that if $L$ is a positive integer, under suitable condition, one gets     \begin{align*}
        \lim\limits_{\substack{q\longrightarrow \infty\\v_L(q-1)=c}}\dfrac{\left|\per(x^L,\M_\ell(q))\right|}{|\M_\ell(q)|}
    &=    \lim\limits_{\substack{q\longrightarrow \infty\\
    v_L(q-1)=c}}\dfrac{\left|\per(x^L,\GL_\ell(q))\right|}{|\GL_\ell(q)|}\\
    &=\sum\limits_{\substack{\lambda\vdash \ell}}\prod\limits_{i=1}^{r}\dfrac{1}{(\lambda_i)^{m_i}m_i!}\dfrac{1}{L^{m_iv_L(q^{\lambda_i}-1)}},
    \end{align*}
    where the sum run over all partitions $\lambda=(\lambda_1^{m_1},\lambda_2^{m_2},\ldots,\lambda_r^{m_r})$ of $\ell$.
    
We prove a few basic results in \cref{sec:prep}. We prove the results on the limiting values in the case of $\M_n(q)$and $\GL_n(q)$ in \cref{sec:gl-m}. The results on $\Sp_{2n}(q)$ and $\Un_n(q)$ are proved in \cref{sec:symporth}.
Throughout,
$\delta_L(p)$ denotes the the multiplicative order of $p$ modulo $L$, and,
$v_L(c)$ denotes the $L$-adic valuation of $c$.

\section{Preparatory materials}\label{sec:prep}
Since the conjugacy classes of $\M_n(q)$ (and $\GL_n(q)$) are parametrized by irreducible polynomials and partitions, see for example \cite{Macdonald81}, one needs to count irreducible polynomials whose roots satisfy a prescribed power condition.
The following result addresses this count.
\begin{lemma}\label{lem:count}
Let $L$ be a prime integer coprime to the characteristic $p$ of the finite field $\F_q$.
Suppose $\d_n$ denotes the number of monic irreducible polynomial $f$ of degree $n$ over the field $\F_q$ such that each root $\alpha$ of $f$ satisfies $\alpha^{e_{n}}=1$, where $e_{\ell}\colon = \dfrac{q^\ell-1}{L^{v_L(q^\ell-1)}}$ for an integer $\ell\geq 1$.
Then
\begin{align*}
    \d_n=\dfrac{1}{n}\left(\sum\limits_{i|n}\mu\left(\dfrac{n}{i}\right)\dfrac{q^i-1}{L^{v_p(q^i-1)}}\right),
\end{align*}
where $\mu$ is the Mobius function.
\end{lemma}
\begin{proof}
    Note that $\d_n$ is the number of elements in the following set
    \begin{align*}
        D_n=\left\{\alpha\in\mathbb F_{q^n}|\alpha^{e_n}=1,\alpha\not\in\mathbb F_{q^i}\text{ for all }i\mid n\right\}.
    \end{align*}\
    For a given $i\mid n$, an element $\alpha\in\mathbb F_{q^i}$ and $\alpha^{e_n}=1$ implies that $\alpha^{\gcd(q^i-1,e_n)=1}$.
    Note that $i\mid n$, so $n=i\cdot n'$, and hence
    \begin{align*}
        v_L(q^n-1)=v_L(q^{in'}-1)=v_L(q^i-1)+v_L(n'`),
    \end{align*}
    and hence $\gcd(q^i-1,e_n)=\gcd\left(q^i-1,\dfrac{(q^i-1)(\sum _{j=0}^{n/i-1} q^{ij})}{L^{v_L(q^i-1)}\cdot L^{v_L(n/i)}}\right)=\dfrac{q^i-1}{L^{v_L(q^i-1)}}$. Then the result follows from the inclusion-exclusion principle.
\end{proof}
For the classification of conjugacy classes of the finite symplectic group, one needs self-reciprocal monic polynomials, which are monic polynomials $f(t)$ such that $f(0)\neq 0$ and $\widetilde{f}(t)\colon=f(0)^{-1}t^{\deg f}f(1/t)=f(t)$; see \cite{Wall1963}.
\begin{lemma}\label{lem:symp-count-root}
    Let $L$ be a prime integer coprime to the characteristic $p$ of the finite field $\F_q$ with $L|q-1$.
    Suppose $\widetilde{\d}_{2n}$ denotes the number of self-reciprocal irreducible monic polynomial $f$ of degree $2n$ over the field $\F_q$ such that each root $\alpha$ of $f$ satisfies $\alpha^{e_{2n}}=1$. 
    Then
    \begin{align*}
        \widetilde{\d}_{2n}=\dfrac{1}{2n}\left(\sum\limits_{\substack{i|n\\i\,\,\text{odd}}}\mu\left({i}\right)\dfrac{q^{n/i}+1}{L^{v_L(q^{n/i}+1)}}\right).
    \end{align*}
\end{lemma}
\begin{proof}
    Using \cite[Theorem 1]{MeynGotz1990}, we get that such a root $\alpha$ os $f$ must satisfy $\alpha^{q^n+1}=1$. 
    Hence, one gets that $\alpha^{\gcd(e_{2n},q^n+1)}=1$.
    Thus, by the inclusion-exclusion principle 
    \begin{align*}
        \widetilde{\d}_{2n}=\dfrac{1}{2n}\left(\sum\limits_{\substack{i|n\\i\,\,\text{odd}}}\mu\left(i\right)\gcd\left(\dfrac{q^{2n/i}-1}{L^{v_L(q^{2n/i}-1)}},q^{n/i}+1\right)\right).
    \end{align*}
    Note that here we are taking the sum over the indices $i$ such that $i$ is odd and $i|n$.
    This happens because each irreducible factor (of degree $\geq2$) of $t^{q^n}+1$ is a self-reciprocal irreducible monic polynomial of degree $2d$, where $d$ divides $n$ such that $n/d$ is odd; see \cite[Theorem 1]{MeynGotz1990}.

    Next, one has $\gcd\left(\dfrac{(q^{n/i}+1)(q^{n/i}-1)}{L^{v_L(q^{2n/i}-1)}},q^{n/i}+1\right)=\dfrac{q^{n/i}+1}{L^{v_L(q^{n/i}+1)}}$, from which the result follows.
\end{proof}
A self-conjugate monic polynomial is a polynomial $f(t)\in\F_{q^2}[t]$ such that $f(0)\neq0$ and $\overline{f}(t)\colon=(f(0)^q)^{-1}t^{\deg f}\sum\limits_{i}a_i^qt^{-i}=f(t)$, where $f(t)=\sum\limits_{i} a_it^i$. 
We need these polynomials while classifying the conjugacy classes of the unitary groups; see \cite{Wall1963}.
\begin{lemma}\label{lem:uni-count-root}
    Let $L$ be a prime integer coprime to the characteristic $p$ of the finite field $\F_{q^2}$ with $L|q-1$.
    Suppose $\overline{\d}_{n}$ denotes the number of self conjugate irreducible monic polynomial $f$ of degree $n$ over the field $\F_{q^2}$ such that each root $\alpha$ of $f$ satisfies $\alpha^{e_{n}}=1$. 
    Then
    \begin{align*}
        \overline{\d}_{n}=\dfrac{1}{n}\left(\sum\limits_{i\mid n}\mu\left({i}\right)\dfrac{q^{n/i}+1}{L^{v_L(q^{n/i}+1)}}\right).
    \end{align*}
\end{lemma}
\begin{proof}
    Since a root $\alpha$ of $f$ also satisfies $\alpha^{q^n+1}=1$, the result follows from a similar treatment as in the proof of \cref{lem:symp-count-root}.
\end{proof}
We finally end this section with the following easy-to-prove lemma.
\begin{lemma}\label{lem:conj}
    Let $f\in \F_q[t]$ be a polynomial, $A\in\M_n(q)$ be a matrix. Then $f^m(A)=A$ iff for any conjugate $B$ of $A$ one has $f^m(B)=B$.
\end{lemma}

\section{Proportion of periodic points in full matrix algebra and the general linear group}\label{sec:gl-m}
Using \Cref{lem:conj}, it is immediate that, in order to decide whether an element $A$ satisfies $f^m(A)=A$, it suffices to consider a representative of its conjugacy class over the algebraic closure $\overline{\F_q}$.

We have the following result describing when a matrix is periodic
\begin{proposition}\label{prop:candidacy}
Let $L$ be a prime integer coprime to the characteristic $p$ of the finite field $\F_q$, and let $f$ denote the map $x\mapsto x^L$. Assume that $\delta_L(p)\mid d$, where $q=p^d$ for some $d\ge 1$.
A matrix $A \in \M_n(q)$ is periodic if and only if all eigenvalues of $A$ are periodic and there is no nilpotent block in the Jordan canonical form of $A$, over an algebraic closure of $\F_q$.
In particular, a matrix $A\in \GL_n(q)$ is periodic if and only if all its eigenvalues are periodic.
\end{proposition}
\begin{proof}
    We work in the field $\overline{\F_q}$, invoking \Cref{lem:conj}. 
    Hence, the matrix $A$ can be taken into its Jordan canonical form $\diag(J_{\lambda_q,m_1},J_{\lambda_2,m_2},\cdots, J_{\lambda_u,m_u})$ for positive integers $m_i$, where
    \begin{align*}
        J_{\lambda,m}=\begin{pmatrix}
            \lambda & 1 &&&&\\
            &\lambda&1&&&\\
            &&\ddots&&&\\
            &&&\lambda&1\\
            &&&&\lambda
        \end{pmatrix}_{m\times m}\in\M_m(\overline{\F_q}).
    \end{align*}
    Since $f^m(\diag(J_{\lambda_q,m_1},J_{\lambda_2,m_2},\cdots, J_{\lambda_u,m_u}))=\diag(J_{\lambda_q,m_1},J_{\lambda_2,m_2},\cdots, J_{\lambda_u,m_u})$ if and only if $f^m(J_{\lambda_i,m_i})=J_{\lambda_i,m_i}$ for all $i$, we may well assume that $A$ up to conjugacy has a single Jordan block.
    When $\lambda=0$, there is nothing to prove as $f^r(J_{0,n})=0$ for some $r>0$; hence, such a matrix can never enter a cycle. 
    So we assume that $\lambda\neq0$ and consequently $J_{\lambda,n}\in\GL_n(q)$.
    Now, note that for any $r>0$
    \begin{align*}
        J_{\lambda,n}^r=
\begin{pmatrix}
\lambda^r & \dbinom{r}{1}\lambda^{r-1} & \dbinom{r}{2}\lambda^{r-2} & \cdots & \dbinom{r}{n-1}\lambda^{r-(n-1)} \\
0 & \lambda^r & \dbinom{r}{1}\lambda^{r-1} & \cdots & \dbinom{r}{n-2}\lambda^{r-(n-2)} \\
0 & 0 & \lambda^r & \cdots & \dbinom{r}{n-3}\lambda^{r-(n-3)} \\
\vdots & \vdots & \vdots & \ddots & \vdots \\
0 & 0 & 0 & \cdots & \lambda^r
\end{pmatrix},
\end{align*}
which implies that for having $J_{\lambda,n}^r=J_{\lambda,n}$ 
\begin{align*}
    \displaystyle{\dbinom{r}{i}}\lambda^{r-i}\equiv\begin{cases}
       \lambda&\text{when }\,i=0\\
       1&\text{when }\,i=1\\
       0&\text{when }\, n-1\geq i\geq 2
    \end{cases}.
\end{align*}
Thus, we get that $p\bigm|\dbinom{r}{i}$ for all $i\geq 2$, and $r\equiv 1\pmod{p}$.
If $r=L^{r'}$ we have that $\delta_p(L)|r'$.

Consider the base-$p$ expansion, $r=1+\sum\limits_{i=1}^{u}r_ip^i$ and $s=\sum\limits_{i=0}^{u}s_ip^i$.
Recalling Lucas theorem, one has
\begin{align*}
    \dbinom{r}{s}\equiv\prod\limits_{i=1}^u\dbinom{r_i}{s_i}\pmod{p}.
\end{align*}
First we get that $r\geq n$; otherwise $\dbinom{r}{n-t}=1$ for some $t\geq 1$.
Write $r=p^{v_p(r-1)}a$ for some $a\not\equiv 0 \pmod{p}$.
Let $a=\sum\limits_{i=0}^\ell a_ip^i$.
Thus we have 
\begin{align*}
    r=1+a_1p^{v_p(r-1)+1}+a_2p^{v_p(r-1)+2}+\ldots=1+\sum\limits_{i=0}^{\ell}a_ip^{v_p(r-1)+i};
\end{align*}
where $v_p(r-1)\geq 1$ as $p|r-1$, and $a_1\neq 0$. Now, if possible, let $n>p^{v_p(r-1)}$. Fix $i_0=1+p^{v_p(r-1)}$, which implies by Lucas theorem that
\begin{align*}
    \dbinom{r}{i_0}\equiv\dbinom{1}{1}\cdot\dbinom{a_1}{1}\neq 0\pmod{p};
\end{align*}
also $2<i_0\leq n-1$. 
We conclude that $n\leq p^{v_p(r-1)}$.
By applying Lucas theorem again, when $n\leq p^{v_p(r-1)}$, one has $\dbinom{r}{i}\equiv0\pmod{p}$ for all $2\leq i\leq n-1$.

Now, let us focus on the power map $x\mapsto x^L$, and we work over the algebraic closure of the underlying field.
A matrix $A=\bigoplus\limits_{i=1}^{\ell}J_{\lambda_i,m_i}$ is periodic- that is $A^{L^r}=A$ if and only if $\lambda_j^{L^r}=\lambda_j$ for all $1\leq j\leq \ell$ and $m_i\leq p^{v_p(L^r-1)}$.
One also has that, if $\mu^{L^r}=\mu$ for some $\mu\in\overline{\F_q}$, then for any $b\in\Z_{\geq 1}$, $\mu^{L^{br}}=\mu$. 
Using the Lifting The Exponent, one has $v_p(L^{br}-1)=v_p(L^r-1)+v_p(b)$. 
Given $n$, one may choose $b$ such that $n<v_p(L^{br}-1)$, where $\lambda^{L^r}=1$ with $\lambda\in\sigma(A)$. The result follows.
\end{proof}
Thus, we need to first find the fixed points under iteration of the power map for a given finite field. 
This quantity is more generally presented and proved in \cite[Theorem 1]{ShaHu2011}.
However, we present a proof for the sake of completeness; also see \cite[Lemma 4.2]{ManesThompson2019} for a different proof of the following result.
\begin{lemma}\label{lem:fix-field}
    Let $L$ be a prime integer coprime to the characteristic $p$ of the finite field $\F_q$, and let $f$ denote the map $x\mapsto x^L$. An element $\alpha\in \F_q$ is periodic if and only if either $\alpha=0$ or $\alpha^d=1$, where $q-1=L^{v_L(q-1)}d$.
\end{lemma}
\begin{proof}
    Let $\alpha$ be periodic. 
    If $\alpha=0$, we are done.
    Else let $\alpha^{L^r}=\alpha$, which implies $\alpha^{L^r-1}=1$ and hence $(\operatorname{ord}(\alpha),L)=1$.
    Since $\alpha$ is an element of the cyclic group $\F_q^\times$, we have $\alpha^{q-1}=1$, which implies the forward part.

    For the converse, if $\alpha=0$, there is nothing to prove. So, let $\alpha^d=1$ where $q-1=L^{v_L(q-1)}d$.
    Since $(L,d)=1$, $L$ is invertible in the group $\Z/d\Z^\times$, and hence there exists $m$, a positive integer such that $L^n=1+md$.
    Then $\alpha^{L^n}=\alpha$; i.e. $\alpha$ is periodic.
\end{proof}
We now discuss two examples; 
\begin{example}\label{exm-m-2}
Consider the algebra $\M_2(q)$ and the power map $x\mapsto x^L$ where $\delta_L(p)|d$ where $q=p^d$. 
We have $\d_1=\dfrac{q-1}{L^{v_L(q-1)}}$ and $\d_2=\dfrac{1}{2}\cdot\left(\dfrac{q^2-1}{L^{v_L(q^2-1)}}-\dfrac{q-1}{L^{v_L(q-1)}}\right)$.
The conjugacy classes of $\M_2(q)$ (under the action of $\GL_2(q)$) along with their sizes are as follows:
\begin{center}
\renewcommand{\arraystretch}{1.4}
\begin{longtable}{cccc}
\hline
Representative &
Conditions &
Number &
Size \\
\hline
\endfirsthead

\hline
Representative &
Conditions &
Number &
Size \\
\hline
\endhead

\hline
\endfoot

\hline
\endlastfoot

$\begin{pmatrix}
\lambda&0\\
0&\lambda
\end{pmatrix}$
&
$\lambda\in\F_q$
&
$q$
&
$1$
\\

$\begin{pmatrix}
\lambda&0\\
0&\mu
\end{pmatrix}$
&
$\lambda,\mu\in\F_q,\ \lambda\neq\mu$
&
$\dfrac{q(q-1)}{2}$
&
$q(q+1)$
\\

$\begin{pmatrix}
\lambda&1\\
0&\lambda
\end{pmatrix}$
&
$\lambda\in\F_q$
&
$q$
&
$q^2-1$
\\

$\begin{pmatrix}
0&-b\\
1&a
\end{pmatrix}$
&
$x^2-ax+b\in\F_q[x]\ \text{irreducible}$
&
$\dfrac{q^2-q}{2}$
&
$q(q-1)$
\\
\caption{Conjugacy classes in $\M_2(q)$}\label{tab:M2}
\end{longtable}
\end{center}
An element $\begin{pmatrix}
    \lambda & \\&\lambda
\end{pmatrix}$ is periodic iff $\lambda^{e_1}=1$ and hence there are $\d_1+1$ such periodic conjugacy classes; they contribute a total of $\d_1+1$ periodic elements.
Similarly the matrices conjugate to $\begin{pmatrix}
    \lambda &\\&\mu
\end{pmatrix}$ - where $\lambda\neq\mu$ - contribute $q(q+1)\dfrac{\d_1(\d_1+1)}{2}$ elements.
If $\lambda=0$, matrices conjugate to $\begin{pmatrix}
    \lambda&1\\&\lambda
\end{pmatrix}$ are not periodic; for $\lambda\neq 0$ they contribute a total of $(\d_1+1)(q^2-1)$ elements.
Finally, the last type of matrices will contribute $q(q-1)\d_2$ elements. 
Hence the total number of matrices in $\M_2(q)$ which are periodic under the map $x\mapsto x^L$ is given by
\begin{align*}
    &\dfrac{q-1}{L^{v_L(q-1)}}+1+\dfrac{1}{2}\left(\dfrac{q-1}{L^{v_L(q-1)}}\right)\left(\dfrac{q-1}{L^{v_L(q-1)}}+1\right)(q^2+q)+\left(\dfrac{q-1}{L^{v_L(q-1)}}\right)(q^2-1)\\
    +&\dfrac{1}{2}\left(\dfrac{q^2-1}{L^{v_L(q^2-1)}}-\dfrac{q-1}{L^{v_L(q-1)}}\right)(q^2-q)\\
    =&\dfrac{1}{2}\cdot \left(\dfrac{q^4}{L^{2v_L(q-1)}}+\dfrac{q^4}{L^{v_L(q^2-1)}}\right)+ \Sigma_{<4}(q),
\end{align*}
where $\Sigma_{<4}$ is a polynomial of degree strictly less than $4$ with rational coefficient.
Thus we get
\begin{align*}
          \lim\limits_{\substack{q\longrightarrow \infty\\v_L(q-1)=c}}\dfrac{\left|\per(x^L,\M_2(q))\right|}{|\M_2(q)|}=\dfrac{1}{2}\left(\dfrac{1}{L^{2v_L(q-1)}}+\dfrac{1}{L^{v_L(q^2-1)}}\right).
\end{align*}
Since the elements contributing to the degree $4$ terms are invertible, we end up getting that
\begin{align*}
          \lim\limits_{\substack{q\longrightarrow \infty\\v_L(q-1)=c}}\dfrac{\left|\per(x^L,\GL_2(q))\right|}{|\GL_2(q)|}=\dfrac{1}{2}\left(\dfrac{1}{L^{2v_L(q-1)}}+\dfrac{1}{L^{v_L(q^2-1)}}\right).
\end{align*}
\end{example}
\begin{example}\label{exm:m-3-q}
    Consider the algebra $\M_3(q)$ and the power map $x\mapsto x^L$ where $\delta_L(p)|d$ where $q=p^d$. 
    We have $\d_1=\dfrac{q-1}{L^{v_L(q-1)}}$, $\d_2=\dfrac{1}{2}\cdot\left(\dfrac{q^2-1}{L^{v_L(q^2-1)}}-\dfrac{q-1}{L^{v_L(q-1)}}\right)$, and $\d_3=\dfrac{1}{3}\cdot\left(\dfrac{q^3-1}{L^{v_L(q^3-1)}}-\dfrac{q-1}{L^{v_L(q-1)}}\right)$.
    Here is the conjugacy class information:
\begin{center}
\begin{center}
\renewcommand{\arraystretch}{1.4}
\begin{longtable}{ccccc}
\hline
Type
& Representative
& Conditions
& Number
& Size \\
\hline
\endfirsthead

\hline
Type
& Representative
& Conditions
& Number
& Size \\
\hline
\endhead

1
&
$\begin{pmatrix}
\alpha&0&0\\
0&\alpha&0\\
0&0&\alpha
\end{pmatrix}$
&
$\alpha\in\F_q$
&
$q$
&
$1$
\\

2
&
$\begin{pmatrix}
\alpha&1&0\\
0&\alpha&1\\
0&0&\alpha
\end{pmatrix}$
&
$\alpha\in\F_q$
&
$q$
&
$q(q^3-1)(q^2-1)$
\\

3
&
$\begin{pmatrix}
\alpha&1&0\\
0&\alpha&0\\
0&0&\alpha
\end{pmatrix}$
&
$\alpha\in\F_q$
&
$q$
&
$(q^3-1)(q+1)$
\\

4
&
$\begin{pmatrix}
\alpha&0&0\\
0&\beta&0\\
0&0&\gamma
\end{pmatrix}$
&
$\alpha,\beta,\gamma\in\F_q$ pairwise distinct
&
$\dbinom{q}{3}$
&
$q^3(q+1)(q^2+q+1)$
\\

5
&
$\begin{pmatrix}
\alpha&0&0\\
0&\alpha&0\\
0&0&\beta
\end{pmatrix}$
&
$\alpha,\beta\in\F_q,\ \alpha\neq\beta$
&
$q(q-1)$
&
$q^2(q^2+q+1)$
\\

6
&
$\begin{pmatrix}
\alpha&1&0\\
0&\alpha&0\\
0&0&\beta
\end{pmatrix}$
&
$\alpha,\beta\in\F_q,\ \alpha\neq\beta$
&
$q(q-1)$
&
$q^2(q^3-1)(q+1)$
\\

7
&
$\begin{pmatrix}
0&0&-\alpha_0\\
1&0&-\alpha_1\\
0&1&-\alpha_2
\end{pmatrix}$
&
$\begin{aligned}
g(t)&=t^3+\alpha_2t^2+\alpha_1t+\alpha_0,\\
& g \text{ irreducible over }\F_q
\end{aligned}$
&
$\dfrac{1}{3}(q^3-q)$
&
$q^3(q^2-1)(q-1)$
\\

8
&
$\begin{pmatrix}
\alpha&0&0\\
0&0&-\alpha_0\\
0&1&-\alpha_1
\end{pmatrix}$
&
$\begin{aligned}
f(t)&=t^2+\alpha_1t+\alpha_0,\\
& f \text{ irreducible over }\F_q
\end{aligned}$
&
$\dfrac{1}{2}q(q^2-q)$
&
$q^3(q^3-1)$
\\\hline
\caption{Conjugacy classes in $\M_3(q)$}
\label{tab:M3}
\end{longtable}
\end{center}
\end{center}
Each type of conjugacy classes have the following contributions;
\begin{align*}
\renewcommand{\arraystretch}{1.4}
\begin{array}{cc|cc}
    \text{Type} & \text{Contribution} & \text{Type} & \text{Contribution}  \\
    \hline
    1 & \d_1+1 & 5 & \d_1(\d_1+1)\cdot q^2(q^3-1)(q+1)\\
    2 & \d_1\cdot q(q^3-1)(q^2-1) & 6 & \d_1^2\cdot q^2(q^3-1)(q+1)\\
    3 & \d_1\cdot(q^3-1)(q+1) & 7 & \d_3 \cdot q^3 (q^2-1)(q-1)\\
    4 & \dbinom{\d_1+1}{3}\cdot q^3(q+1)(q^2+q+1)& 8 & (\d_1+1)\d_2\cdot q^3(q^3-1)
\end{array}.    
\end{align*}
Hence, the total number of matrices that are periodic under the power map is 
\begin{align*}
    \dfrac{1}{6}\cdot \dfrac{q^9}{L^{3v_L(q-1)}}+\dfrac{1}{3}\cdot\dfrac{q^9}{L^{v_L(q^3-1)}}+\dfrac{1}{2}\cdot\dfrac{q^9}{L^{v_L(q-1)}L^{v_L(q^2-1)}} +\Sigma_{<9}(q),
\end{align*}
where $\Sigma_{<9}$is a polynomial of degree strictly less than $9$ with rational coefficients. 
Thus, we get by a similar argument in the previous example, 
\begin{align*}
          \lim\limits_{\substack{q\longrightarrow \infty\\v_L(q-1)=c}}\dfrac{\left|\per(x^L,\M_3(q))\right|}{|\M_3(q)|}
    &=    \lim\limits_{\substack{q\longrightarrow \infty\\
    v_L(q-1)=c}}\dfrac{\left|\per(x^L,\GL_3(q))\right|}{|\GL_3(q)|}\\
    &=\dfrac{1}{6}\cdot \dfrac{1}{L^{3v_L(q-1)}}+\dfrac{1}{3}\cdot\dfrac{1}{L^{v_L(q^3-1)}}+\dfrac{1}{2}\cdot\dfrac{1}{L^{v_L(q-1)}L^{v_L(q^2-1)}}.
\end{align*}
\end{example}
In view of \Cref{prop:candidacy}, if $A\in \M_n(q)$ is periodic, the Jordan block corresponding to $0$ is diagonal. 
Also \Cref{exm-m-2} and \Cref{exm:m-3-q} suggest that the (restricted) limiting value of $\dfrac{\left|\per(x^L,\M_\ell(q^n))\right|}{|\M_\ell(q^n)|}$ and $\dfrac{\left|\per(x^L,\GL_\ell(q^n))\right|}{|\GL_\ell(q^n)|}$ for positive integer $\ell\geq 2$. 
Indeed, this is the case, and we have the following result.
\begin{proposition}\label{prop:lim-reduction}
    For a positive integer $\ell\geq2$ consider the algebra $\M_\ell(q)$ and the power map $x\mapsto x^L$ where $\delta_L(p)|d$, $q=p^d$. 
    Then
    \begin{align*}
              \lim\limits_{\substack{q\longrightarrow \infty\\v_L(q-1)=c}}\dfrac{\left|\per(x^L,\M_\ell(q))\right|}{|\M_\ell(q)|}
    &=    \lim\limits_{\substack{q\longrightarrow \infty\\
    v_L(q-1)=c}}\dfrac{\left|\per(x^L,\GL_\ell(q))\right|}{|\GL_\ell(q)|}.
    \end{align*}
    Further, the limiting value is contributed by the classes of regular semisimple invertible matrices.
\end{proposition}
\begin{proof}
    It is enough to show that the number of elements with $0$ as an eigenvalue contributes a polynomial of degree strictly less than $\ell^2$.
    It is enough to work for a specific \emph{type} of matrix. 
    Fix a \textit{multiset} $\Lambda=\{(d_0,\lambda_0),(d_1,\lambda_1),(d_2,\lambda_2),\ldots, (d_\ell,\lambda_\ell)\}$ where $d_0=1$, $\lambda_0=1^{m_0}$, $d_i\geq 1$ are positive integers, and $\lambda_i$ is a partition of $m_i$ for all $\ell\geq i\geq 1$; further $\sum\limits_{(d_0,\lambda_i)\in\Lambda} d_im_i=n$.
    Consider the set
    \begin{align*}
        \mathscr A_\Lambda=\left\{A\in \M_n(q):\delta_A=1\right\},
    \end{align*}
    where $\delta_A=1$ if and only if the characteristic polynomial is of the form $t^{m_0}\prod\limits_{i} f_i(t)^{m_i}$ and the combinatorial data attached to $A$ is of the form $\{(f_i,\lambda_i):i\}$ and $\deg f_i=d_i$; note that the choice of $f_i$ depends on the matrix $A$ although the degree sequence remains the same. 
    Then $A$ is conjugate to a diagonal block matrix given as $\bigoplus\limits_{i} J_i\oplus \mathbf 0$, where
    \begin{align*}
        J_i=\begin{pmatrix}
            C_{f_i}&\I_{d_i}&&\\
            &C_{f_i}&&\\
            &&\ddots&\\
            &&&C_{f_i}
        \end{pmatrix},
    \end{align*}
    is a $m_i\deg f_i\times m_i\deg f_i$ matrix, and $\mathbf{0}$ is the zero matrix of size $m_0\times m_0$. 
    Let $J\colon = \bigoplus\limits_{i}J_i.$
    Hence, one gets that for $A\in\mathscr A$
    \begin{align*}
    \left|X\in\GL_n(q):AX=XA\right|=|\Zen_{\GL}(J)|\cdot |\GL_{m_0}(q)|.
    \end{align*}
    Let us first work out the general case, and we will specialize in $J$ soon.
    Consider $X\in \GL_s(q)$ such that the minimal polynomial of $X$ is given by $\prod\limits_{i=1}^{u}\prod\limits_{j=1}^{n_i}f_{i,j}^{e_{i,j}}$ where $\deg f_{i,j}=d_i$.
    Suppose $X$ has Jordan form
    \begin{align*}
        \bigoplus\limits_{i=1}^{u}\bigoplus\limits_{j=1}^{n_i}\left(J_{\lambda_{i,j,1}}^{l_{i,j,1}}\oplus J_{\lambda_{i,j,2}}^{l_{i,j,2}}\oplus\ldots\oplus J_{\lambda_{i,j,k_{i,j}}}^{l_{i,j,k_{i,j}}}\right),
    \end{align*}
    where $J_{\lambda_{i,j,k}}$ is the Jordan block of dimension $\lambda_{i,j,k}$ relative to $f_{i,j}(t)$; further the combinatorial data attached to $X$ is given by $\left\{(f_{i,j},\lambda^{f_{i,j}})\mid i,j\right\}$ where $\lambda^{f_{i,j}}=\left(\lambda_{i,j,1}^{l_{i,j,1}},\lambda_{i,j,2}^{l_{i,j,2}}\ldots,\lambda_{i,j,k_{i,j}}^{l_{i,j,k_{i,j}}}\right)$ with $\lambda_{i,j,1}<\lambda_{i,j,2}<\ldots<\lambda_{i,j,k_{i,j}}=e_{i,j}$. 
    Then
    \begin{align*}
    &|\Zen_{\GL}(X)|=q^{\gamma} \left|
    \prod\limits_{i=1}^{u}
    \prod\limits_{j=1}^{n_i}
    \prod\limits_{k=1}^{k_{i,j}}
    \left(\GL_{l_{i,j,k_{i,j}}}(q^{d_i})\right)
    \right| 
    \end{align*}
    where 
    \begin{align*}
        \gamma=\sum\limits_{i=1}^{u}d_i\left(\sum\limits_{j=1}^{n_i}\left(2\sum\limits_{a<b}\lambda_{i,j,a}l_{i,j,a}l_{i,j,b}+\sum\limits_{k=1}^{k_{i,j}}(\lambda_{i,j,k}-1)l_{i,j,k}^2\right)\right);
    \end{align*}
    see \cite[Theorem 2.2.8]{FranceshiLiebeckOBrienBook}.

    Coming back to the specific case, we get that the number of elements in $\mathscr A$ is given by
    \begin{align*}
        \dbinom{\d_{d_1}}{n_1}\dbinom{\d_{d_2}}{n_2}\cdots\dbinom{\d_{d_u}}{n_u}\dfrac{|\GL_\ell(q)|}{|\GL_{m_0}(q)|\cdot|\Zen_{\GL}(J)|},
    \end{align*}
    where $J\in \GL_{\ell-m_0}(q)$.
    Note that the conjugacy class size is invariant for all $A'$ such that $A\in \mathscr A$.
    Hence, the highest power of $q$ appearing in the expression $|\mathscr A|$ is at most
    \begin{align*}
        \mathfrak H_q=&\sum
        \limits_{i=1}^{u} n_id_i 
        +\ell^2-m^2_0-\gamma -\sum\limits_{i=1}^u\sum\limits_{j=1}^{n_i}\sum\limits_{k=1}^{k_{i,j}}
        d_i
        l_{i,j,k_{ij}}^2.
    \end{align*}
    Note that 
    \begin{align*}
        &\gamma+\sum\limits_{i=1}^u\sum\limits_{j=1}^{n_i}\sum\limits_{k=1}^{k_{i,j}}
        d_i
        l_{i,j,k_{ij}}^2\\
        =&\sum\limits_{i=1}^{u}d_i\left(\sum\limits_{j=1}^{n_i}\left(2\sum\limits_{a<b}\lambda_{i,j,a}l_{i,j,a}l_{i,j,b}+\sum\limits_{k=1}^{k_{i,j}}(\lambda_{i,j,k}-1)l_{i,j,k}^2\right)\right)
        -
        \sum\limits_{i=1}^u\sum\limits_{j=1}^{n_i}\sum\limits_{k=1}^{k_{i,j}}
        d_i
        l_{i,j,k_{ij}}^2\\
        =&\sum\limits_{i=1}^{u}d_i\left(\sum\limits_{j=1}^{n_i}\left(2\sum\limits_{a<b}\lambda_{i,j,a}l_{i,j,a}l_{i,j,b}+\sum\limits_{k=1}^{k_{i,j}}\lambda_{i,j,k}l_{i,j,k}^2\right)\right)\\
        =&\sum\limits_{i=1}^{u}d_i\left(\sum\limits_{j=1}^{n_i}\left(\sum\limits_{a,b}\operatorname{min}\{\lambda_{i,j,a},\lambda_{i,j,b}\}l_{i,j,a}l_{i,j,b}\right)\right).
    \end{align*}
    Hence 
    \begin{align*}
        \mathfrak H_q-\ell^2&=\sum\limits_{i=1}^{u}n_id_i-m_0^2-\sum\limits_{i=1}^{u}d_i\left(\sum\limits_{j=1}^{n_i}\left(\sum\limits_{a,b}\operatorname{min}\{\lambda_{i,j,a},\lambda_{i,j,b}\}l_{i,j,a}l_{i,j,b}\right)\right)\\
        &\leq \sum\limits_{i=1}^u n_id_i-m_0^2-\sum\limits_{i=1}^u d_i\left(\sum\limits_{j=1}^{n_i}\sum\limits_{a}  \lambda_{i,j,a}l_{i,j,a}^2\right)&(\text{since }\lambda_{i,j,a}\geq 1)\\
        &=-m_0^2-\sum\limits_{i=1}^u d_i\sum\limits_{j=1}^{n_i}\left\{\left(\sum\limits _{a}\lambda_{i,j,a}l_{i,j,a}^2\right)-1\right\},
    \end{align*}
    which implies that $\mathfrak H_q=\ell^2$ if and only if $m_0=0$ and $\lambda_{i,j,1}=l_{i,j,1}=1$ and for all $a>1,$ $l_{i,j,a}=0$.
    Consequently, only regular semisimple invertible elements contribute to the limiting value.
    \end{proof}
\begin{theorem}\label{thm:GL}
    For a positive integer $\ell\geq2$ consider the algebra $\M_\ell(q)$ and the power map $x\mapsto x^L$ where $\delta_L(p)|d$, $q=p^d$. 
    Then
    \begin{align*}
        \lim\limits_{\substack{q\longrightarrow \infty\\v_L(q-1)=c}}\dfrac{\left|\per(x^L,\M_\ell(q))\right|}{|\M_\ell(q)|}
    &=    \lim\limits_{\substack{q\longrightarrow \infty\\
    v_L(q-1)=c}}\dfrac{\left|\per(x^L,\GL_\ell(q))\right|}{|\GL_\ell(q)|}\\
    &=\sum\limits_{\substack{\lambda\vdash \ell}}\prod\limits_{i=1}^{r}\dfrac{1}{(\lambda_i)^{m_i}m_i!}\dfrac{1}{L^{m_iv_L(q^{\lambda_i}-1)}},
    \end{align*}
    where the sum run over all partitions $\lambda=(\lambda_1^{m_1},\lambda_2^{m_2},\ldots,\lambda_r^{m_r})$ of $\ell$.
\end{theorem}
\begin{proof}
    In view of \Cref{prop:lim-reduction}, it is enough to work with the regular semisimple classes. 
    A matrix $X\in \GL_\ell(q)$ is regular semisimple if and only if the combinatorial data attached to it is given by $\{(f_{i,j},1): i, 1\leq j\leq m_i\}$ where $f_{i,j}$ are irreducible polynomials of degree $d_i$ for all $j$ and $\sum\limits d_i m_i=\ell$.
    This determines a partition of $\ell$, more precisely the partition $\lambda=(d_1^{m_1},d_2^{m_2},\ldots, d_r^{m_r})$.
    The number of elements in such a conjugacy class is given by
    $\dfrac{|\GL_\ell(q)|}{\prod\limits_{i=1}^{r}(q^{d_i}-1)^{m_i}}$.
    The number of such classes is given by
    \begin{align*}
        &\dbinom{\d_{d_1}}{m_1}\dbinom{\d_{d_2}}{m_2}\ldots\dbinom{\d_{d_r}}{m_r}\\
        =&\prod\limits_{i=1}^r\dfrac{\d_{d_i}(\d_{d_i}-1)\ldots (\d_{d_i}-m_i+1)}{m_i!}\\
        =&\prod\limits_{i=1}^r \dfrac{1}{(d_i)^{m_i}m_i!}\left(\dfrac{q^{d_i}-1}{L^{v_L(q^{d_i-1})}}+\cdots\right)\left(\dfrac{q^{d_i}-1}{L^{v_L(q^{d_i-1})}}+\cdots-1\right)
        \ldots \left(\dfrac{q^{d_i}-1}{L^{v_L(q^{d_i-1})}}-m_i+1\right)\\
        =&\prod\limits_{i=1}^r\dfrac{1}{(d_i)^{m_i}m_i!}\left\{\dfrac{q^{d_im_i}}{L^{m_iv_L(q^{d_i}-1)}}+\Sigma_{<d_im_i}\right\},
    \end{align*}
    where $\Sigma_{<d_im_i}$ is a polynomial with rational coefficient of degree strictly less than $d_im_i$.
    Hence summing over all such partitions we get the result.
\end{proof}

\section{Periodic points in Symplectic and Unitary groups}\label{sec:symporth}

\begin{theorem}\label{thm:symp}
    For a positive integer $\ell\geq 2$ consider the finite symplectic group $\Sp_{2\ell}(q)$ and the power map $x\mapsto x^L$ where $L$ is a prime and $\delta_L(p)|d$, $q=p^d$. Then 
    \begin{align*}
        \lim\limits_{\substack{q\longrightarrow\infty\\v_L(q-1)=c}}\dfrac{|\per(x^L,\Sp_{2\ell}(q))|}{|\Sp_{2\ell}(q)|}= \sum\limits_{\lambda\vdash \ell}\prod\limits_{i=1}^r\dfrac{1}{(\lambda_i)^{m_i}m_i!}\dfrac{1}{L^{m_iv_L(q^{\lambda_i}+1)}}\prod\limits_{a=1}^{s}\dfrac{1}{(\lambda_a')^{m_a'}m_a'!}\dfrac{1}{L^{m_a'v_L(q^{\lambda_i'}-1)}},
    \end{align*}
    where the sum runs over all partitions $\lambda=(\lambda_1^{m_1},\ldots,\lambda_r^{m_r},\lambda_1'^{m_1'},\ldots,\lambda_s'^{m_s'})$ of $\ell$.
\end{theorem}
\begin{proof}
We first get a lower bound using the regular semisimple conjugacy classes, and later work on the equality part.
    Recall that a regular semisimple element of $\Sp_{2\ell}(q)$ does not have any eigenvalues $\pm 1$.
    Hence, a matrix $X\in \Sp_{2\ell}(q)$ is regular semisimple if and only if the combinatorial data attached to it is of the form $\{(f_{i,j},1)\mid i,1\leq j\leq m_i\}\cup\{g_{a,b}g^*_{a,b}\mid a,1\leq b\leq m'_a\}$, where
    \begin{enumerate}
        \item $f_{i,j}$ is a self reciprocal irreducible monic polynomial of degree $2d_i$
        \item $g_{a,b}\neq g_{a,b}^*$ is monic irreducible polynomial of degree $d'_a$
        \item $\sum\limits d_im_i +\sum d'_am'_a=\ell$.
    \end{enumerate}
    The number of elements in such a conjugacy class is given by $\dfrac{|\Sp_{2\ell}(q)|}{\prod\limits_{i=1}^{r}(q^{d_i}+1)\prod\limits_{a=1}^s(q^{d_a}-1)}$; see \cite[page 36]{Wall1963}.
    Using the degrees of the polynomial, we get two partitions $(2d_1^{m_1},2d_2^{m_2},\ldots, 2d_r^{m_r})$ and $(d_1'^{m_1'},d_2^{m_2'},\ldots,d_{s}'^{m_s'})$. 
    The number of such conjugacy classes is given by 
    \begin{align*}
        &\prod\limits_{i=1}^{r}\dbinom{\widetilde{\d}_{2d_i}}{m_i}\cdot\prod\limits_{a=1}^{s}\dbinom{\d_{d'_a}}{m_a'}\\
        =&\prod\limits_{i=1}^{r}\dfrac{1}{(2d_i)^{m_i}m_i!}\left\{\dfrac{q^{d_im_i}}{L^{m_iv_L(q^d_i+1)}}+\Sigma_{<d_im_i}\right\}\prod\limits_{a=1}^{s}\dfrac{1}{(d_a')^{m_a'}m_a'!}\left\{\dfrac{q^{d_a'm_a'}}{L^{m_av_L(q^{d_a'}-1)}}+\Sigma_{d_a'm_a'}\right\},
    \end{align*}
    where $\Sigma_{<d_im_i}$ and $\Sigma_{<d_a'm_a'}$ are polynomials of degree strictly less than $d_im_i$ and $d_a'm_a'$ respectively; the argument is similar to \cref{thm:GL}. Taking the limit, we get a lower bound, the same as the expression in the theorem. It is then enough to achieve a similar result, which is in parallel to \cref{prop:lim-reduction}.

    Assume $L$ to be odd. 
    For keeping the notation as per \cite{Wall1963}, we use a slightly different notation for the parts of a partition.
    Consider the type of matrices (periodic under the $L$-th power) with combinatorial data
    \begin{align*}
        \left\{(t\pm 1,\lambda^{t\pm1})\right\}\cup\left\{(f_{i,j},\lambda^{f_{i,j}})\mid 1\leq j\leq m_i\right\}\cup\{(g_{a,b}g_{a,b}^*,\lambda^{g_{a,b}})\mid 1\leq b\leq m_a\},
    \end{align*}
    where \begin{enumerate}
        \item $\lambda^{t\pm 1}$ is a signed symplectic partition (i.e., by which a partition of some natural number such that the odd parts have even multiplicity, together with a choice of sign for each even part) corresponding to the polynomial $t\pm 1$,
        \item $f_{i,j}(t)\neq t\pm 1$ is a self-reciprocal irreducible monic polynomial of degree $2d_i$ for all $1\leq j\leq m_i$, $\lambda^{f_{i,j}}$ is a partition,
        \item $g_{a,b}(t)\neq g_{a,b}^*(t)$ is a irreducible monic polynomial of degree $d'_a$ for all $1\leq b\leq m_a'$, $\lambda^{g_{a,b}}$ is a partition,
        \item $\sum \limits |\lambda^{t\pm 1}|+\sum\limits_{i,j}2d_i|\lambda^{f_{i,j}}|+\sum\limits_{a,b}2d_a'|\lambda^{g_{i,j}}|=2\ell$.
    \end{enumerate}
    Given a matrix with the above combinatorial data, the number of elements in each conjugacy class is determined by the work of Wall \cite{Wall1963}, which is $\dfrac{|\Sp_{2\ell}(q)|}{\prod\limits_{h}B(h,\lambda^h) }$; where for $1\leq i \leq n$
    \begin{align*}
        A(h,u)=\begin{cases}
            |\Un_{m_u(\lambda^h)}(q^{\deg h/2})|&\text{if}\,\,h(t)=\widetilde{h}(t)\neq t\pm 1\\
            |\GL_{m_u(\lambda^h)}(q^{\deg h})|^{1/2}&\text{if}\,\, h(t)\neq\widetilde{h}(t)\\
            |\Sp_{m_u(\lambda^h)}(q)|&\text{if}\,\,h(t)=t\pm 1,u\equiv 1\pmod{2}\\
            q^{m_u(\lambda^h)/2}|\Or^\epsilon_{m_u(\lambda^h)}(q)| &\text{if}\,\,h(t)=t\pm 1,u\equiv 0\pmod{2}
        \end{cases},
    \end{align*}
    where $\epsilon\in\{\pm\}$ which depend on the sign of $i$ in $\lambda^{t\pm 1}$ and 
    \begin{align*}
        B(h,\lambda^h)=q^{\deg h\left(\sum_{u<v}um_u(\lambda^h)m_v(\lambda^h)+(\sum_{u}(u-1)m_u(\lambda^h)^2)/2\right)}\prod\limits_{u=1}^{2\ell}A(h,u).
    \end{align*}
    In the above formulae $m_u(\lambda^h)$ denotes the number of parts in $\lambda^h$ of size $u$, and $A(h,u)$ is assumed to be $1$ when $m_u(\lambda^h)=0$.

    Then such a \emph{type} of matrices produces a constituent of the form
    \begin{align*}\label{eq:symp}
        \dfrac{1}{\prod\limits_{h}B(h,\lambda^h)}\prod\limits_{i=1}^r\dbinom{\widetilde{\d}_{2d_i}}{m_i}\prod\limits_{a=1}^s\dbinom{\d_{d_a'}}{m_a'}.
    \end{align*}
    Given that there is a polynomial $h$ such that $|\lambda^h|>1$, we claim that the limiting value of the above expression under $q\longrightarrow \infty$ is zero, where $\delta_L(q-1)=c$. Note that
    \begin{align*}
        \prod\limits_{i=1}^r\dbinom{\widetilde{\d}_{2d_i}}{m_i}\prod\limits_{a=1}^s\dbinom{\d_{d_a'}}{m_a'}=\mu q^{\left(\sum\limits_{i=1}^rd_im_i+\sum\limits_{a=1}^sd_a'm_a'\right)} + \Sigma,
    \end{align*}
    for some $\mu\in\mathbb Q$ and a polynomial $\Sigma$ of degree strictly less than the degree $\sum d_im_i+\sum d_a'm_a'$.
    Because we work with a particular type of matrix, we fix a signed partition. 
    Since for any choice of $\epsilon\in\{\pm 1\}$ and a positive integer $\mu$, one has $|\Or^\epsilon_{2\mu}(q)|=2q^{\mu^2-\mu}(q^\mu-\epsilon)\prod\limits_{j=1}^{\mu-1}(q^{2j}-1)$, the highest power of $q$ contributed into the constituent is $2\mu^2-\mu=(2\mu)^2/2-(2\mu)/2$.
    In case of $\Or^\epsilon_{2\mu+1}(q)$ the highest power of $q$ is $2\mu^2+\mu=(2\mu+1)^2/2-(2\mu+1)/2$.
    Thus indeed one has for any $\mu$, the highest power of $q$ in $|\Or^\epsilon_\mu(q)|$ is given by $\mu^2/2-\mu/2$.
    Hence, we get that the highest power of $q$ in the polynomial determined by $\prod\limits_{h}B(h,\lambda^h)$ is
    \begin{align*}
\widetilde{\mathfrak H}_q=&\sum\limits_{\epsilon=\pm 1}\left(\sum\limits_{u<v}um_u(\lambda^{t+\epsilon})m_v(\lambda^{t+\epsilon})+\dfrac{1}{2}\sum\limits_{u}(u-1)m_u(\lambda^{t+\epsilon})^2)\right)\\
+&\sum\limits_{\epsilon=\pm 1}\left(\sum\limits_{u\equiv 0\,(2)}\dfrac{m_u(\lambda^{t+\epsilon})}{2}\left(\dfrac{m_u(\lambda^{t+\epsilon})^2}{2}-\dfrac{m_u(\lambda^{t+\epsilon})}{2}\right)+\sum\limits_{u\equiv 1\,(2)}\left(\dfrac{m_u(\lambda^{t+\epsilon})^2}{2}+\dfrac{m_u(\lambda^{t+\epsilon})}{2}\right)\right)\\
+&\sum\limits_{i=1}^r 2d_i\sum\limits_{j=1}^{m_i}\left(\sum\limits_{u<v}um_u(\lambda^{f_{i,j}})m_v(\lambda^{f_{i,j}})+\dfrac{1}{2}\sum\limits_{u}(u-1)m_u(\lambda^{f_{i,j}})^2)+\sum\limits_{u}\dfrac{d_i}{2}m_u(\lambda^{f_{i,j}})^2\right)\\
+&\sum\limits_{a=1}^{s}2d_a'\sum\limits_{s=1}^{m_a'}\left(\sum\limits_{u<v}um_u(\lambda^{g_{a,b}})m_v(\lambda^{g_{a,b}})+\dfrac{1}{2}\sum\limits_{u}(u-1)m_u(\lambda^{g_{a,b}})^2)+\sum\limits_{a}d_a'm_u(\lambda^{g_{a,b}})^2\right).
\end{align*}
If there is one $h$ such that $|\lambda(h)|>1$, then either there exists $u\geq 2$ such that $m_u(\lambda^h)\geq 1$ or there exists $m_1(\lambda^h)>1$.
In both scenarios, $\widetilde{\mathfrak H}_q$ is strictly greater than $\sum d_im_i+\sum d_a'm_a'$.
Taking a similar study in case $L$ is even, the result follows by considering the limit $q\longrightarrow\infty$.
\end{proof}
\begin{theorem}\label{thm:unitary}
    For a positive integer $\ell\geq 2$ consider the finite symplectic group $\Un_{\ell}(q)$ and the power map $x\mapsto x^L$ where $L$ is a prime and $\delta_L(p)|d$, $q=p^d$. Then 
    \begin{align*}
        \lim\limits_{\substack{q\longrightarrow\infty\\v_L(q-1)=c}}\dfrac{|\per(x^L,\Un_{\ell}(q))|}{|\Un_{\ell}(q)|}= \sum\limits_{\lambda\vdash \ell}\prod\limits_{i=1}^r\dfrac{1}{(\lambda_i)^{m_i}m_i!}\dfrac{1}{L^{m_iv_L(q^{\lambda_i}+1)}}\prod\limits_{a=1}^{s}\dfrac{1}{(\lambda_a')^{m_a'}m_a'!}\dfrac{1}{L^{m_a'v_L(q^{\lambda_i'}-1)}},
    \end{align*}
    where the sum runs over all partitions $\lambda=(\lambda_1^{m_1},\ldots,\lambda_r^{m_r},\lambda_1'^{m_1'},\ldots,\lambda_s'^{m_s'})$ of $\ell$.
\end{theorem}
\begin{proof}
    The proof is similar to that of \cref{thm:symp}. 
    However, one needs to take care of the different conjugacy class sizes here; this is available at \cite[page 34]{Wall1963}.
\end{proof}
\printbibliography
\vspace{2em}
\end{document}